\documentclass{article}
\usepackage[utf8]{inputenc}
\usepackage{amsmath}
\usepackage{amsthm}
\usepackage{amsfonts}
\usepackage{amssymb}
\usepackage{csquotes}
\usepackage{authblk}

\usepackage[left=2cm,right=2cm,
    top=2cm,bottom=2cm,bindingoffset=0cm]{geometry}

\usepackage[
backend=biber,
style=numeric,
sorting=ynt
]{biblatex}

\let\leq\leqslant
\let\geq\geqslant

\title{Isometry groups of formal languages for generalized Levenshtein distances \footnote{The work of the author was supported by the RSF, project no. 22-11-00075}}
\author{Vladimir Yankovskiy}
\affil{Faculty of Mechanics and Mathematics of Moscow State University, Moscow 119991 Russia, Leninskie Gory, MSU.}
\affil{Moscow Center for Fundamental and Applied Mathematics, Russia}
\affil{vladimir\_yankovskiy@mail.ru}
\date{}

\begin{document}

\maketitle

\begin{abstract}

This article is a partial answer to the question of which groups can be represented as isometry groups of formal languages for generalized Levenshtein distances. Namely, it is proved that for any language the modulus of the difference between the lengths of its words and the lengths of their images under isometry for an arbitrary generalized Levenshtein distance that satisfies the condition that the weight of the replacement operation is less than twice the weight of the removal operation is bounded above by a constant that depends only on the language itself.
From this, in particular, it follows that the isometry groups of formal languages with respect to such metrics always embed into the group $\Pi_{n=1}^\infty S_{n}$.
We also construct a number of examples showing that this estimate is, in a certain sense, unimprovable.


\end{abstract}

\newtheorem{theorem}{Theorem}
\newtheorem{lemma}{Lemma}
\newtheorem{corollary}{Corollary}
\newtheorem{definition}{Definition}
\newtheorem{proposition}{Proposition}

\section{Introduction}

This paper is a partial answer to the question of which groups can be represented as isometry groups of formal languages for generalized Levenshtein distances. Namely, the following theorem is proved:

\begin{theorem}
Let $L$ be an arbitrary formal language, $d$ be an arbitrary generalized Levenshtein distance satisfying the condition that the weight of the replacement operation is less than twice the weight of the removal operation. Then there exists $m \in \mathbb{N}$ such that $||\phi(w)| - |w|| \leq m$
\end{theorem}

Note that for generalized Levenshtein distancess with a replacement weight greater than or equal to twice the insertion weight, this statement is not true.

It follows from Theorem 1, in particular, that the isometry group of any formal language with respect to the Levenshtein metric embeds in $\Pi_{n=1}^\infty S_{n}$.

A number of examples are also constructed, demonstrating that this estimate is, in a certain sense, unimprovable, moreover, for languages with different growths.

\begin{definition}
\textbf{Growth} of $L$ is a function of $n \mapsto |\{w \in L| |w| \leq n\}|$
\end{definition}

\begin{theorem}
Let $G$ be an arbitrary finite group, $d$ an arbitrary generalized Levenshtein distance. Then there are $n \in \mathbb{N}$ and language $L \subset \{0; 1\}^{24n}$ such that $|L| = n$, $Isom_d(L) \cong G$ and $d(u, v) \in \{4, 6\}$ for any two distinct words $u, v \in L$.
\end{theorem}

Here and below, $A^k$ denotes the set of all words of length $k$ over the alphabet $A$.

\begin{theorem}
Let $G_1, G_2, ... $ be a countable sequence of arbitrary finite groups, $d$ be an generalized Levenshtein distance. Then there exists a language $L \subset \{0; 1\}^*$ with growth $O(n)$ such that $Isom_d(L) \cong \Pi_{n=1}^\infty G_n$.
\end{theorem}

\begin{theorem}
 Let $d$ be an arbitrary generalized Levenshtein distance. Then for any integer $k \geq 2$, there exists a language $L$ over an alphabet of $k$ characters with growth $\Theta(k^{\sqrt{n}})$ such that $Isom_d(L) \cong S_k^\infty \times \Pi_{n=1}^\infty S_{k^n}$.
\end{theorem}

The last fact is interesting because the isometry group obtained by Theorem 1 is maximal for all languages.

Also, special attention is paid to regular languages:

\begin{definition}
\textbf{Regular language} is a formal language that can be obtained from finite languages by applying a finite number of union operations ($U \cup V$), product ($UV = \{uv| u \in U, v \in V \}$) and Kleene stars ($L^* = \bigcup_{n = 0}^\infty L^n$).
\end{definition}

It is proved in the article that their isometry groups are also sufficiently diverse and the given bound for them is unimprovable:

\begin{theorem}
Let $G$ and $H$ be arbitrary finite groups, $d$ be an arbitrary generalized Levenshtein distance. Then there is a regular language $L \subset \{0; 1\}^*$ such that $Isom_d(L) \cong G \times H^\mathbb{N}$.
\end{theorem}

\begin{theorem}
Let $d$ be an arbitrary distance from the family of generalized Levenshtein distances. Then there exists a regular language $L$ such that $Isom_d(L) \cong \Pi_{n=1}^\infty S_{2n}$.
\end{theorem}

The last fact is interesting because the isometry group obtained by Theorem 1 is maximal for all languages, not necessarily regular ones.

Here and below, $S_n$ denotes a symmetric group on $n$ elements, $C_n$ denotes a cyclic group of order $n$, $\Pi$ denotes the Cartesian product of groups.

The isometry groups of finite languages have been studied before.

For example, in \cite{mrk} it is proved that for an arbitrary finite alphabet $A$, the isometry group $A^n$ with respect to the Hamming distance is isomorphic to $S_{|A|}^n \times S_n$.

Another paper on a similar topic is \cite{ruthladser}, where it is proved that for an arbitrary finite alphabet $A$ and $2 \leq k_1 < k_2$ the isometry group of the language $\bigcup_{k = k_1}^{k_2} A^k$ with respect to the <<internal Levenshtein distance>> (the minimum number of operations of insertions, deletions and replacements of characters that transform one word into another in such a way that all <<intermediate words>> lie in the source language) is isomorphic to $S_{|A|}^n \times C_2$

In the same article, the isometry groups of languages, including infinite ones, are studied with respect to the family of generalized Levenshtein distances. Their definitions and main properties will be given in Section 1.

The work consists of 9 sections (including introduction):
\begin{itemize}
\item Section 2 formulates the definition of a family of generalized Levenshtein metrics, and also classifies isometry groups of one-character languages.
\item Section 3 proves Theorem 1 and constructs a counterexample for the case when the weight of the replacement is greater than or equal to twice the weight of the insertion.
\item In Section 4, we carry out the preparatory work necessary for the proof of Theorem 2.
\item Section 5 proves Theorem 2.
\item Section 6 proves Theorem 3.
\item Section 7 provides a proof
\item Section 7 proves Theorem 4.
\item Section 8 proves Theorem 5.
\item Section 9 proves Theorem 6.
\end{itemize}

\section{Generalized Levenshtein distances}

\begin{definition}
Let $A$ be a finite alphabet and the words $u, v \in A^*$.
Then \textbf{generalized Levenshtein distance} with insertion weight $\gamma$ and replacement weight $\theta$ between $v$ and $u$ is $$lev_{\gamma, \theta}(u, v) = \min \{ \gamma n + \theta m | u \text{ can be translated into } v \text{ by n insertions or deletions and m replacements} \}$$
\end{definition}

Obviously, for $\gamma, \theta > 0$ $lev_{\gamma, \theta}$ is a distance on $A^*$.

The most studied special cases of the generalized Levenshtein distance are:
\begin{itemize}
    \item[$lev_{1, 1}$] classical Levenshtein metric -- the minimum number of insertions, deletions, or substitutions required to tranform one word into another (first discussed in \cite{levenshtein})
    \item[$lev_{1, 2}$] is the minimum number of insertions or deletions required to transform one word into another.
    \item [$lev_{n, 1}$] for $n \to \infty$ converges pointwise to the Hamming metric, the minimum number of substitution operations required to transform one word into another.
\end{itemize}

Also, there is an alternative way to specify the generalized Levenshtein distance by the recursive formula:

\begin{proposition}[\cite{wagfish}]
$$lev_{\gamma, \theta}(a, b) = \begin{cases} \gamma |a| & \quad |b| = 0 \\ \gamma |b| & \quad |a| = 0 \\ lev_{\gamma, \theta}(a.tail, b.tail) & \quad a.head = b.head \\ \min(\theta + lev_{\gamma, \theta}(a. tail, b.tail), \gamma + lev_{\gamma, \theta}(a.tail, b), \gamma + lev_{\gamma, \theta}(b.tail, a)) & \quad a. head \neq b.head \end{cases}$$

where $a.tail$ is the suffix of $a$ containing all of its characters except the first one, and $a.head$ is the first element of $a$.
\end{proposition}

In particular, from this formula, as well as the invariance of the Levenshtein distance under <<reflection>> of words, it is true that $lev_{\gamma, \theta}(uxv, uyv) = lev_{\gamma, \theta}(x, y )$.

The following inequalities also always hold:
\begin{itemize}
    \item $lev_{\gamma, \theta}(u, v) \leq (\theta - \gamma) \min(|u|, |v|) + \gamma \max(|u|, |v|) $
    \item $lev_{\gamma, \theta}(u, v) \geq \gamma ||u| - |v||$, and equality is achieved if and only if the shorter word is contained in the longer one as a subsequence.
\end{itemize}

\begin{definition}
Let $(M_1, d_1)$ and $(M_2, d_2)$ be metric spaces.
\newline
We will call a bijection $\phi:M_1 \to M_2$ \textbf{homothety} if there exists $t \in \mathbb{R}$ such that for any words $u, v \in \mathbb{L}$ $ d_2(\phi(u), \phi(v)) = t d_1(u, v)$.
\newline
We will call a bijection $\phi:M_1 \to M_2$ \textbf{isometry} if for any words $u, v \in \mathbb{L}$ $d_2(\phi(u), \phi(v)) = d_1(u, v)$.
\end{definition}

The set $Isom_d(M)$ of all isometries of the metric space $(M, d)$ into itself forms a composition group. Moreover, the isometry group of a metric space is always isomorphic to the isometry group of its image under homothety.

\begin{proposition}
Let $d$ be a generalized Levenshtein metric and $L$ an arbitrary language. Then there exists $\theta \in (0; 2]$ such that the trivial mapping $(L, d)$ to $(L, lev_{1, \theta})$ is a homothety.
\end{proposition}
\begin{proof}
As $\theta$ we can take $\min(\frac{\theta_0}{\gamma_0}, 2)$
\end{proof}

We will denote the distance $lev_{1, \theta}$ as $lev_\theta$, and the isometry group of the language $L$ with respect to it as $Isom_{\theta}(L)$.

The class of isometry groups of languages over a one-element alphabet is rather small.

\begin{proposition}
Let $L \subset \{a\}^*$, $\theta \in (0; 2]$. Then $Isom_\theta(L)$ is either trivial or isomorphic to a cyclic group of order 2.
\end{proposition}

\begin{proof}
$l: a^n \mapsto n$ is an isometry of $(\{a\}^*, lev_{\theta})$ onto the metric space $(\mathbb{N}, d(m, n) = |n-m| )$.
\newline
From this we can conclude that $Isom(L) \cong Isom(l(L))$.
\newline
Suppose $N \subset \mathbb{N}$, $|N| \geq 2$ (everything is obvious for $|N|=1$). Let $n_0, n_1, ... $ be elements of $N$ sorted in ascending order, $\phi$ be an isometry from $N$.
\newline
$\phi(n_0) \in \{n_0, n_{|N|}\}$. Otherwise, if $\phi(n_0) = n_i \neq n_{|N|}$, then $$|n_{i+1} - n_{i-1}| = |\phi^{-1}(n_{i+1}) - \phi^{-1}(n_{i-1})| <$$ $$< |\phi^{-1}(n_{i+1}) - n_0| + |n_0 - \phi^{-1}(n_{i-1})| = $$ $$ = |(n_{i+1} - n_i| + |n_{i} - n_{i-1}| = |n_{i+1} - n_{i-1}|$$
It's impossible.
\newline
Let $\phi(n_0) = n_0$. Let us prove by induction that $\phi$ is trivial:
\begin{itemize}
\item[Base:] For $N = 0$ $\phi(n_0) = n_0$ (no other items).
\item[Step:] suppose that $\phi(n_i) = n_i$ for all $i < k$. Then $n_k$ is the only nearest element to $n_{k-1}$ not contained in $\{n_0, ... , n_{k-1}\}$. Therefore, $\phi(n_k)$ is the only nearest element to $\phi(n_{k-1})=n_{k-1}$ not contained in $\phi(\{n_0, ... , n_{ k-1}\}) = \{n_0, ... , n_{k-1}\}$, i.e. $\phi(n_k)=n_k$.
\end{itemize}
Let $\phi(n_0) = n_{|N|}$. Let us prove by induction that $\phi: n_k \mapsto n_{|N|-k}$.
\newline
\begin{itemize}
\item[Base:] For $N = 0$ $\phi(n_0) = n_0 = n_{|N|}$ (no other elements).
\item[Step:] suppose that $\phi(n_i) = n_{|N|-i}$ for all $i < k$. Then $n_k$ is the only nearest element to $n_{k-1}$ not contained in $\{n_0, ... , n_{k-1}\}$. Therefore, $\phi(n_k)$ is the only nearest element to $\phi(n_{k-1})=n_{|N|-k+1}$ not contained in $\phi(\{n_0, .. . , n_{k-1}\}) = \{n_{|N|}, ... , n_{|N|-k+1}\}$, i.e. $\phi(n_k)=n_{ |N|-k}$.
\end{itemize}
Thus, there can be no other isometries, except for these two (and the second one is far from always realized).
\newline
Hence $Isom(L)$ is either trivial or isomorphic to a cyclic group of order 2.

\end{proof}

However, already for the two-element alphabet, the isometry groups are more complicated, which will be demonstrated in the following sections.

\section{Proof of Theorem 1}

\begin{lemma}[\cite{higman}]
There is no infinite language over a finite alphabet in which no word occurs as a subsequence of another.
\end{lemma}

We will denote by $M(L)$ the set of all minimal words in the language $L$ with respect to the inclusion order as a subsequence.

\begin{lemma}
Let $L$ be a formal language, $\theta \in (0; 2)$. Then if $Isom_\theta(L)$ acts transitively on $L$, then $L$ is finite.
\end{lemma}
\begin{proof}
Let $L$ be an infinite language. $M(L)$ is finite by Lemma 1. Therefore, by the Dirichlet principle, $L$ has an infinite sublanguage with a unique minimal word. Let's denote it as $L_0$ and the only minimal word in it as $w_0$.
\newline
Language $\{w \in L_0| |w| > \frac{2}{2 - \theta}|w_0|\}$ is also infinite. So, arguing similarly, it also has an infinite sublanguage with a single minimal word. Let's denote it as $L_1$ and the only minimal word in it as $w_1$.
\newline
Language $\{w \in L_1| |w| > 2|w_1|\}$ is also infinite. So, arguing similarly, it also has an infinite sublanguage with a single minimal word. Let's denote it as $L_2$, and the only minimal word in it as $w_2$.
\newline
Now suppose that $Isom_\theta(L)$ acts transitively on $L$. Then there exists $\phi \in Isom_\theta(L)$ such that $\phi(w_1) = w_0$. Then the chain of inequalities is fulfilled:

$$lev_\theta(w_0, w_1) + lev_\theta(w_1, w_2) = lev_\theta(w_0, w_2) = lev_\theta(\phi(w_0), \phi(w_2)) \leq $$ $$ \leq \max(|\phi(w_0)|, |\phi(w_2)|) + (\theta - 1) \min(|\phi(w_0)|, |\phi(w_2)|) \leq $$ $$\leq \theta |w_0| + \max(lev_\theta(w_0, \phi(w_0)), lev_\theta(w_0, \phi(w_2))) + (\theta - 1)\min(lev_\theta(w_0, \phi(w_0) )), lev_\theta(w_0, \phi(w_2))) = $$ $$ = \theta |w_0| + \max(lev_\theta(w_1, w_0), lev_\theta(w_1, w_2)) + (\theta - 1)\min(lev_\theta(w_1, w_0), lev_\theta(w_1, w_2)) = $$ $$ = \theta |w_0| + |w_2| - |w_1| + (\theta - 1)(|w_1| - |w_0|) < lev_\theta(w_0, w_1) + lev_\theta(w_1, w_2)$$

Contradiction.
\end{proof}

Let now $m = \max\{|\phi(v)| -|v|| v \in M(L), \phi \in Isom_\theta(L)\}$ (this set is finite by Lemmas 1 and 2). Let us prove by contradiction that $||w| - |\phi(w)|| \leq m$.

Let $||w| - |\phi(w)|| >m$. Without loss of generality, we assume that $|w| < |\phi(w)|$. Now let $u \in M(L)$ be contained in $w$ as a subsequence.

Then

$$lev_\theta(w, u) = lev_\theta(phi(w), phi(u)) \geq |\phi(w)| - |\phi(u)| > |w| + m - |\phi(u)| \geq $$ $$\geq |w| + |\phi(u)| - |u| - |\phi(u)| = |w| - |u| = lev_\theta(w, u)$$

Contradiction.$\square$

\begin{corollary}
The isometry group of any formal language $L$ embeds in $\Pi_{n=1}^\infty S_{n}$.
\end{corollary}
\begin{proof}
Let $n_1, n_2, ...$ be the orders of the orbits of the natural action of $Isom(L)$ on $L$ (the orbits are finite by Lemma 2). Since $Isom(L)$ acts effectively on $L$, it embeds in $\Pi_{i=1}^\infty S_{n_i} \leq \Pi_{n=1}^\infty S_{n}$ .
\end{proof}

\begin{proposition}
There is a regular language $L \subset \{0, 1\}^*$ such that the isometry group $Isom_2(L) \cong D_\infty$, and it acts transitively on $L$.
\end{proposition}
\begin{proof}
Consider a regular language $L = 1^* \cup 0^*$ and a bijection $\phi: L \to \mathbb{Z}$ given by the formulas
$$\phi(0^n) = n$$
$$\phi(1^n) = -n$$
for all $n \in \mathbb{N}_0$.
\newline
Let us show that $\phi$ is an isometry of the spaces $(L, lev_2)$ and $(\mathbb{Z}, d(x, y) = |x - y|)$.
Really

$$lev_2(0^n, 0^m) = |m - n|$$
$$lev_2(1^n, 1^m) = |m - n|$$
$$lev_2(0^n, 1^m) = m + n$$

for all $n, m \in \mathbb{N}_0$.
\newline
Moreover, the isometry group of the metric space $(\mathbb{Z}, d(x, y) = |x - y|)$ acts transitively on it and is isomorphic to $D_\infty$.
\end{proof}

\section{Word stretching}

We will denote the Hamming distance as $h$.

\begin{definition}
The \textbf{word stretching} is the operation $st: A^* \times A^* \to A^*$ given recursively:

$$st(\Lambda, w_2) = \Lambda$$
$$st(w_1 a, w_2) = st(w_1, w_2)w_2a$$
for all $w_1, w_2 \in A^*, a \in A$
\end{definition}

\begin{lemma}
Let the words $w_1, w_2 \in A^*$, where $|w_1| = |w_2|$.
\newline
Let $k \in \mathbb{N}$, where $k > h(w_1, w_2)$.
\newline
Let the symbols $a, b \in A$, and $a \neq b$.
\newline
Let $\theta \in (0; 2]$.
\newline
Then the equality holds: $$lev_\theta(st(w_1, a^k b a^k), st(w_2, a^k b a^k)) = h(w_1, w_2)$$
\end{lemma}
\begin{proof}
Let us show that there is a minimal <<path>> from $st(w_1, a^k b a^k)$ to $st(w_2, a^k b a^k)$ that does not contain insertions or deletions.
\newline
Indeed, suppose that the shortest <<path>> contains a deletion. $k$ deletions cannot contain this <<path>> (otherwise it will not be the shortest one). This means that the <<segment>> of the word between this deletion and one of the insertions closest to it (on one side or the other) is shifted by a distance less than $k$.
\newline
At the same time, at positions that are not multiples of $k + 1$ (numbering starts from 1), there are no symbols different from $a$, and at positions comparable to $k + 1$ modulo $2k + 2$ there are no symbols different from b .
\newline
So, if there are $t$ significant symbols on this <<segment>>, then there are also $t-1$ <<central>> $b$ symbols on it that turned out to be shifted. It would take $2(t - 1)$ replacement operations to put them in order. This means that the total cost of operations on the interval (including the initial deletion and insertion) will be at least $2(t - 1)\theta + 2$. In this case, element-by-element replacement of all significant characters on the segment without any insertions and deletions will cost $t\theta \leq 2(t - 1)\theta + 2$.
\newline
So, gradually changing the pairs <<insert / delete>> to replace significant characters, we get the desired <<path>>.
\end{proof}

\section{Proof of Theorem 2}

\begin{lemma}
Let $\Gamma(V, E)$ be a finite simple cubic graph. Then there exists a language $\{w_v\}_{v \in V} \subset \{0; 1\}^{|E|}$ such that

$$h(w_u, w_v) = \begin{cases} 4 & \quad \{u, v\} \in E \\ 6 & \quad (u, v) \not\in E \end{cases}$$
\end{lemma}

\begin{proof}
Let $E = \{e_0, ... , e_m\}$. Let's define $i:V\times E \to \{0, 1\}^*$ as follows:
\newline
$$i(v, e) = \begin{cases} 1 & \quad v \in e \\ 0 & \quad v \not\in e \end{cases}$$
\newline
Let's define the word $w_v = i(v, e_0)i(v, e_1)...i(v, e_{|E|})$.
\newline
It is easy to see that the Hamming distances between words turned out to be as required.
\end{proof}

\begin{theorem}{\em(\cite{frucht})}
A group is finite if and only if it is isomorphic to the automorphism group of some cubic graph.
\end{theorem}

Let $G$ be a finite group, $\Gamma(E, V)$ be the cubic graph corresponding to it by Theorem 7, $L_0$ be the language constructed for $\Gamma$ by Lemma 2, $L_1 = st (L_0, 1^7 0 1^7)$.
\newline
Since $diam(L_1) = 6$, by Lemma 1, $Isom_\theta(L_1) \cong G$. The length of words in $L_1$ is equal to $16|E| = 24|V| = 24|L_1|$.
$\square$

\section{Proof of Theorem 3}

Let $G_1, G_2, ... $ be a countable sequence of finite groups.
\newline
$L_1, L_2, ...$ are the corresponding languages constructed by Corollary 1.
\newline
We construct the languages $L_1^\prime, L_2^\prime, ...$ using recursion:

$$L_0^\prime = \{\Lambda\}$$

$$L_{n+1}^\prime = (01)^{l_n + 7} L_{n+1}$$

where $l_n$ is the length of words in $L_{n}^\prime$.

It is easy to see that the isometry group $L_n^\prime$ coincides with the isometry group $L_n$. Moreover, for any words $w_n \in L_n^\prime$ and $w_m \in L_m^\prime$, $m > n$ $lev_\theta(w_n, w_m) = l(L_m^\prime) - l(L_n^\prime)$. This equality is true because, to get a shorter word from a longer one, it is enough to remove everything except the $(01)^{l(L_n)}$ prefix, and then select the correct character from each $01$ block.
\newline
It follows from this that for any sequence $\phi_1, \phi_2, ...$ of isometries of languages $L_1^\prime, L_2^\prime, ...$ it is true that $\phi: w_n \mapsto \phi(w_n )$ for all $n \in \mathbb{N}$, $w_n \in L_n^\prime$ is an isometry of $\bigcup_{n=1}^\infty L_n^\prime$.
\newline
Moreover, $\bigcup_{n=1}^\infty L_n^\prime$ has no other isometries.
\newline
Indeed, $\Lambda$ goes into itself as the only word without neighbors at a distance of $4$ or $6$. So the length of words in this language (that is, the distance to $\Lambda$) is an invariant.
\newline
In other words, $Isom_\theta(\bigcup_{n=1}^\infty L_n^\prime) = \Pi_{n=1}^\infty Isom_\theta(L_n^\prime) = \Pi_{n=1 }^\infty Isom_\theta(L_n) = \Pi_{n=1}^\infty G_n$.
\newline
$\bigcup_{n=1}^\infty L_n^\prime$ has growth $O(n)$ since for any $n \in \mathbb{N}$ the length of all words in $L_n^\prime$ is greater $24|L_n|$, which means $|\{w \in L| |w| \leq k\}| \leq 1 + \frac{k}{24}$. $\square$

\section{Proof of Theorem 4}

Let $A = \{a_1, ..., a_k\}$ be an arbitrary alphabet.
Let's build languages $L_1, L_2, ...$ recursively:

$$L_0 = \Lambda$$

$$L_{n} = (a_1...a_k)^{l(k, n - 1)}st(A^{k^n}, a_1^{k^{n + 1}} a_2 a_1^{ k^{n + 1}})$$

where $l(k, n)$ is the length of $L_{n}$.
\newline
Note that $l(k, n) = l(k, n - 1) + 2k^{n}(k^{n + 1} + 2) = O(k^{2n})$. In this case, $|L_n| = k^{k^n}$.
Hence, the language $\bigcup_{n=1}^\infty L_n$ has growth $\sum_{l(k, t) \leq n} k^{k^t} = \Theta(k^{\sqrt{n }})$
\newline
Moreover, it is easy to see that the isometry group $L_n$ is isomorphic to $S_k^{k^n} \times S_{2^n}$. Moreover, the distance between any words $w_n \in L_n^\prime$, $w_m \in L_m$, where $m > n$, is equal to $l(L_m) - l(L_n)$, because to get a shorter words from a longer one, it is enough to remove everything except the prefix $(a_1...a_k)^{l(L_n)}$, and then select the correct character from each $a_1...a_k$.
\newline
It follows from this that for any sequence $\phi_1, \phi_2, ...$ of isometries of languages $L_1, L_2, ...$ it is true that $\phi: w_n \mapsto \phi(w_n)$ for all $n \in \mathbb{N}$, $w_n \in L_n$ is an isometry of $\bigcup_{n=1}^\infty L_n$.
\newline
Moreover, $\bigcup_{n=1}^\infty L_n$ has no other isometries.
\newline
Indeed, $\Lambda$ goes into itself as the only word without neighbors at a distance of $1$. So the length of words in this language (that is, the distance to $\Lambda$) is an invariant.
\newline
In other words, $Isom_\theta(\bigcup_{n=1}^\infty L_n) = \Pi_{n=1}^\infty Isom_\theta(L_n) = S_k^\infty \times \Pi_{n=1 }^\infty S_{k^n}$. $\square$.

\section{Proof of Theorem 5}

\begin{lemma}
Let $L \subset \{0, 1\}^n$, $\theta \in (0; 2]$. Then $Isom_\theta(L((01)^n)^*)\cong Isom(L )^\mathbb{N}$.
\end{lemma}

\begin{proof}
Let $u, v \in L$ and $p, q \in \mathbb{N}$. Let us show that
$$lev_\theta(u(01)^{pn}, v(01)^{qn}) = \begin{cases} 2n|p-q| & \quad p \neq q \\ d(u, v) & \quad p = q \end{cases}$$
\newline
Indeed, in the first case, a longer word is obtained from a shorter one by removing everything superfluous. In the second case, it is enough to perform transformations only on different prefixes.
\newline
From this, in particular, it follows that any transformation of the form $\phi: u(01)^{kn} \mapsto \psi(k)(u)(01)^{kn} $, where $\psi$ -- - an arbitrary function $\mathbb{N} \to Isom_\theta(L)$, is an isometry of $(L((01)^{n})^*, lev_\theta)$.
\newline
To verify that there are no other isometries, note that $u(01)^{kn}$ has exactly $2|L|$ "neighbours" at a distance of $2tn$, for $t \leq k$ and exactly $ |L|$ for $t > k$.

\end{proof}

Using Theorem 2, we can construct uniform languages $L_1 \subset \{0, 1\}^n$ and $L_2 \subset \{0, 1\}^m$ such that $Isom_\theta(L_1) \cong G$, while $Isom_\theta(L_2) \cong H$.
\newline
By Lemma 3 $Isom_\theta(L_2((01)^m)^*) \cong H^\mathbb{N}$.
\newline
Now consider the language $L_1 \cup (01)^{n + m}L_2((01)^m)^*$. It is easy to see that $lev_\theta(u, (01)^{n+m}v(01)^ {mk}) = 2(n + (k+2)m)$.
\newline
Hence, for every $\phi \in Isom_\theta(L_1)$ and $\psi \in Isom_\theta(L_2((01)^{mk})^*)$ the map $\chi$ taking all $u \in L_1$ into $\phi(u)$, and all $(01)^{n+m}v$, where $v \in L_2((01)^{mk})^*$, --- in $(01)^{n+m}\psi(v)$, is an isometry of $L_1 \cup (01)^{n+m}L_2((01)^{mk})^*$.
\newline
Moreover, there can be no other isometries, since the word $v$ from the new language belongs to $L_1$ if and only if it has no "neighbors" at a distance greater than $n$ but less than $n + 2m$.
\newline
So $Isom_\theta(L_1 \cup (01)^{n+1}L_2((01)^{mk})^*) \cong G \times H^\mathbb{N}$. Moreover, $L_1 \cup (01)^{n+1}L_2((01)^{mk})^*$ is regular by construction. $\square$

\section{Proof of Theorem 6}

Consider the language $L = (010)^*(110)(010)^* \cap (\{0, 1\}^6)^*$. Let us show that for arbitrary two words $u, w \in L$ such that $|u| - |v| = 6$, $v$ can be obtained from $u$ with exactly $6$ deletions. Note that at positions dividing $3$ (we assume that the numbering of positions starts from $0$), each word can contain only one unit. Let these special units in the words $u$ and $v$ be at positions $3i$ and $3j$, respectively. Then, if $i = j$, it suffices to remove the suffix of length $6$ (of the form $010010$) from $u$. Otherwise, you must first remove the subword $u_{3i}u_{3i+1}u_{3i + 2} = 110$ from $u$, and then from the subword $u_{3j}u_{3j+1}u_{3j +2}u_{3j+3} = 0100$ (numbering in the new order after deletion) remove $u_{3j}$ and $u_{3j+2}u_{3j+3}$. Thus we get $v$.
\newline
From the above it follows by induction that for $u, v \in L$
\newline
$$lev_\theta(u, v) = \max(||u|-|v||, 2)$$. Thus we see that the distance between words does not depend on anything other than their length. This means that for any sequence of permutations $\sigma_i$ of elements $L \cap A^{6i}$ the map $\phi: u \mapsto \sigma_{\frac{|u|}{6}}(u)$ is an isometry .
\newline
At the same time, the absence of other isometries follows from the fact that each word $u \in L$ has exactly $\frac{2|u|}{3}$ neighbors at a distance of $6$.
\newline
That is, since $|L \cap A^{6i}| = 2i$, $Isom_\theta(L) \cong \Pi_{n=1}^\infty S_{2n}$. $\square$

\section{Thanks}

I express my gratitude to Anton Klyachko and Alexander Olshansky for valuable comments on my work, as well as to Alexei Talambutsa for information about Higman's paper \cite{higman}, which contains the lemma used by me in the proof of Theorem 1.

\printbibliography[title={Bibliography}] 

\end{document}